\documentclass[graybox]{svmult}

\usepackage{amssymb}
\usepackage{mathrsfs}
\usepackage{amsmath}
\usepackage{url}
\usepackage{stackrel,color} 
\usepackage{courier}
\usepackage{graphicx}
\usepackage[color,matrix,arrow,all,cmtip]{xy} 
\usepackage{multicol} 
\usepackage{verbatim} 
\usepackage{hyperref}
\usepackage{multicol}  
\usepackage[dvipsnames]{xcolor}
\usepackage{enumerate} 
\usepackage{bookmark} 
\usepackage{breqn}

\catcode`~=11 
\newcommand{\urltilde}{\kern -.15em\lower .7ex\hbox{~}\kern .04em}  
\catcode`~=13 

\makeatletter
\newcommand*\bigcdot{\mathpalette\bigcdot@{2.0}}
\newcommand*\bigcdot@[2]{\mathbin{\vcenter{\hbox{\scalebox{#2}{$\m@th#1\bullet$}}}}}
\makeatother

\renewcommand{\star}{{}^{*}}

\newcommand{\Stab}{\mathop{\mathrm{Stab}}\nolimits}

\newcommand{\Res}{\mathop{\mathrm{Res}}\nolimits}
\newcommand{\Ind}{\mathop{\mathrm{Ind}}\nolimits}

\newcommand{\Rep}{\mathop{\mathrm{Rep}}\nolimits}

\makeindex             
\allowdisplaybreaks 

\begin{document}
\title*{Generalized iterated wreath products of cyclic groups and rooted trees correspondence}
\author{Mee Seong Im and Angela Wu} 
\institute{Mee Seong Im \at Department of Mathematical Sciences, United States Military Academy, West Point, NY 10996 USA  \newline 
\email{meeseongim@gmail.com}
\and Angela Wu \at Department of Mathematics, University of Chicago, Chicago, IL 60637 USA
\newline 
 \email{wu@math.uchicago.edu}}
%
%
\titlerunning{Iterated wreath products of cyclic groups and rooted trees}
\maketitle


\abstract{Consider the generalized iterated wreath product $\mathbb{Z}_{r_1}\wr \mathbb{Z}_{r_2}\wr \ldots \wr \mathbb{Z}_{r_k}$ where $r_i \in \mathbb{N}$.   
We prove that the irreducible representations for this class of groups are indexed by a certain type of rooted trees. 
This provides a Bratteli diagram for the generalized iterated wreath product, a simple recursion formula for the number of irreducible representations, and a strategy to calculate the dimension of each irreducible representation. We calculate explicitly fast Fourier transforms (FFT) for this class of groups, giving literature's fastest FFT upper bound estimate. 
}

\section{Introduction}\label{section:introduction}     
Representations of groups appear naturally in nature, more often than groups themselves. They appear in the form of a linear representation, a permutation representation, and automorphisms of an algebra, a group, a variety or scheme, or a manifold. 
For example, one can study functions on the circle $S^1$, which could be thought of as a group under addition, which form representations of $S^1$. Such functions could also be thought of as periodic functions on the set $\mathbb{R}$ of real numbers, and the decomposition of the space of functions on $S^1$ is known as the theory of Fourier series.  
One can also study the additive group $\mathbb{R}$ of real numbers acting on itself under addition. Then one may ask how the function space of $\mathbb{R}$ decompose under the action of the group of real numbers; this is the study of Fourier transform.

A cyclic group may be thought of as the set of rotational symmetries of a regular polygon and of a generalized wreath product  $\mathbb{Z}_{r_1}\wr \mathbb{Z}_{r_2}\wr \ldots \wr \mathbb{Z}_{r_k}$ as the automorphisms of a corresponding complete rooted tree generated by cyclic shifts of the children of each node.  
With applications to functions on rooted trees, pixel blurring (cf. \cite{Stankovic-Moraga-Astola}, \cite{Chang-Thesis},  \cite{Mirchandani99multiresolution-analysis}, \cite{Mirchandani99awreath}, \cite{Holmes-mathematical-foundations}, \cite{Holmes-signal-processing}), nonrigid molecules in molecular spectroscopy (cf. \cite{balasubramanian1979enumeration}, \cite{MR585739}, \cite{milot2001energy}, \cite{schnell2010understanding}), and visual information processing (cf. \cite{leyton2003generative}, \cite{borsa2015wreath}), we generalize Orellana-Orrison-Rockmore's manuscript \cite{MR2081042}. 
Denoting the iterated wreath product as $W(\vec{r}|_k):= \mathbb{Z}_{r_1}\wr \ldots \wr \mathbb{Z}_{r_k}$ (see Section~\ref{subsection:wreath-products}), 
we show that the equivalence classes of irreducible representations of the iterated wreath products $W(\vec{r}|_k)$ are indexed by classes of labels on the vertices of the complete $\vec{r}|_k$-ary trees (see Section~\ref{subsection:trees}) of height $k$ (Proposition~\ref{prop:one-to-one-equiv-irrep-orbits-labels}).  

Let $G$ be a finite group and let $V$ be a vector space over the set $\mathbb{C}$ of complex numbers. Let $GL(V)$ be the general linear group on $V$, and let $\rho:G\rightarrow GL(V)$ be a representation of $G$, i.e., $\rho$ is a group homomorphism. 
We say two representations  $\rho:G\rightarrow GL(V)$ and $\eta:G\rightarrow GL(W)$   are equivalent, and write $\rho \sim \eta$, if there exists a vector space isomorphism $f:V\rightarrow W$ such that $f\circ \rho(g)=\eta(g)\circ f$ for all $g\in G$. 
 We denote by $\widehat{G}$ the set of equivalence classes of irreducible representations of $G$.   
We say that  
$\mathcal{R}$ is a \textit{traversal} for $G$ if $\mathcal{R} := \mathcal{R}_G \subset \widehat{G}$ contains one irreducible representation for each isomorphism class in $\widehat{G}$. 
As a basic consequence of representation theory, the equality $\sum_{\rho \in \mathcal{R}} \dim(\rho)^2 = |G|$ holds, where the sum is over all irreducible representations in $\mathcal{R}$. 

We denote by $[n]:=\{ 1,2,\ldots, n\}$ the set of integers from $1$ to $n$, and we denote the set of length $\ell$ words with letters in $[n]$ by 
$$
[n]^\ell := \left\{ x_1 x_2 \cdots x_\ell: x_i \in [n] \right\}. 
$$

Now given a subgroup $H \leq G$, we write $\Ind_H^G:\Rep(H)\rightarrow \Rep(G)$ to be the induction functor from the category of representations of $H$ to the category of representations of $G$. 
That is, given a representation $\eta\in \widehat{H}$, $\eta:H\rightarrow GL(V)$, 
 we write $\Ind_H^G \eta = \mathbb{C}[G]\otimes_{\mathbb{C}[H]}V$, the {\it induced representation} of $G$ from $\eta$ with dimension $[G:H] \cdot \dim \eta$. 
We also have the dual construction to induction, which is called restriction.  Given a subgroup $H$ of $G$,
$\Res_H^G:\Rep(G)\rightarrow \Rep(H)$ is the restriction functor from the category of representations of $G$ to the category of representations of $H$, i.e., 
given a representation $\rho$ of $G$, we obtain the {\em restricted representation} $\Res_H^G\rho$ of $H$ by restricting $\rho$ to $H$. 
The induction and restriction functors are related by Frobenius reciprocity. 
We refer the reader to \cite{MR0202859} for an explicit and elegant discussion on the duality of induction and restriction.

For $x = x_1 \cdots x_{r_k} \in [h]^{r_k}$, define 
\begin{equation}\label{eqn:d_x}
d_x := \min \{ i \in \mathbb{N} : x^i = x \}, \mbox{ where } 
x^ i = (x_1 \cdots x_{r_k})^i := x_{i+1} \cdots x_{r_k} x_1 \cdots x_i.
\end{equation}
Note that $d_x | r_k$ for any $x\in [h]^{r_k}$. We write $x^G=\{x^g:g\in G \}$, 
the orbit of $x$ under $G$. In the case $G=\mathbb{Z}_r$, then $i\in \mathbb{Z}_r$ acts on $x$ by $i\cdot x=x^i$, cyclically rotating the letters in the word $x$.

 We now state our first theorem, which generalizes Theorem 2.1 in \cite{MR2081042}:  
\begin{theorem}\label{theorem:irreps-symm}
Suppose that $\mathcal{R} = \{ \rho_1,\ldots, \rho_h\}$ is a traversal for the iterated wreath product $W(\vec{r}|_{k-1})$. Let $J \subseteq [h]^{r_k}$ denote a set of $\mathbb{Z}_{r_k}$-orbit representatives of $[h]^{r_k}$  such that $ [h]^{r_k}  = \displaystyle{ \bigsqcup_{x \in J} x^{\mathbb{Z}_{r_k}}}$. Then a traversal for $W(\vec{r}|_{k})$ is given by: 
		\begin{equation}
		 \mathcal{R}_{W(\vec{r}|_k)} =  \left\{ \Ind_{W(\vec{r}|_{k-1}) \rtimes 	\mathbb{Z}_{d_x}}^{W(\vec{r}|_k)} 
		(\rho_{x_1} \otimes \cdots \otimes \rho_{x_{r_k}} \otimes \tau) : x \in J, \tau \in \widehat{\mathbb{Z}_{d_x}} \right\}. 
		\end{equation}
	
\end{theorem}

 Now, an efficient algorithm for applying a discrete Fourier transform is called a fast Fourier transform (FFT). 
 For a finite group $G$, we denote by $T(G)$ the maximum number of computations required to compute $\{ \widehat{f}(\rho) : \rho \in \mathcal{R}_G \}$ over all complex-valued functions $f: G \rightarrow \mathbb{C}$ on $G$, where $\widehat{f}$ is the Fourier transform of $f$ at $\rho$, i.e., it is the matrix 
 \[ 
 \widehat{f}(\rho) = \sum_{g\in G} f(g)\rho(g). 
 \] 

For the class of finite abelian groups $G$, the Cooley-Tukey algorithm given in \cite{MR0178586} and \cite{diaconis1980average}, combined with techniques provided in  \cite{Rader} and \cite{MR0243724}, give an order of $O(|G| \log |G|)$ operation bound on the FFT computation time, where the $O$-notation is some universal constant. 
Rockmore in   \cite{MR1053228} provides the fastest algorithm to date in literature for abelian group extensions: 

\begin{theorem}[Lemma 5 and Theorem 4, \cite{MR1053228}]\label{thm:Rockmore-FFT}
Let $ K \trianglelefteq G$ be a normal subgroup of $G$ and assume $G/K$ is abelian.	Let $\rho \in \widehat{G}$. Then there exists a subgroup $H$ with $K \leq H \leq G$ and $\widetilde{\eta} \in \widehat{H}$ such that 
	\begin{itemize}
	\item $\eta = \widetilde{\eta} |_K$ is irreducible, 
	\item $\Ind_H^G \widetilde{\eta} = \rho$, 
	\item $H:= \{ g \in G: \rho^{(g)} \sim \rho \}$, the inertia group of $\rho$ in $G$, where $\rho^{(g)}(h)=\rho(g^{-1}hg)$ for all $h\in G$, and 
	\item if $G = \bigsqcup \limits_{i \in [G:H]} s_i H$, then $\rho = \widetilde{\eta }^{(1)} \otimes \widetilde{\eta}^{(s_2)} \otimes \cdots \otimes \widetilde{\eta}^{(s_{[G:H]})}$, where $\widetilde{\eta}^{(s)}(h) = \widetilde{\eta}(s^{-1} h s)$. 
	\end{itemize}
\end{theorem}

 The representation $\rho^{(g)}$ is called a {\em conjugate representation} of $\rho$. 
By Theorem~\ref{thm:Rockmore-FFT}, we see that it suffices to take a traversal $\mathcal{R}_K$ of $K$ in order to find a complete traversal of $G$. Then we need to construct the set of extensions of $\eta$ to its inertia group $H_\eta$ for each representation $\eta \in \mathcal{R}_K$. Finally, we need to build up the induced representation of $G$ from each extension. 
Applying \cite{MR1053228} to give an upper bound on the number of operations needed to compute a fast Fourier transform of an iterated wreath product, 
 we present an explicit running time of fast Fourier transforms for $W(\vec{r}|_k)$, thus proving a tighter upper bound estimate than Theorem 1 in \cite{MR1053228}:

\begin{theorem}\label{theorem:TG-upper-bound}
For $f: K \wr \mathbb{Z}_r \rightarrow \mathbb{C}$, we have 
\begin{dmath}
T(K \wr \mathbb{Z}_r) =  
	r \cdot T(K^r) +	\sum\limits_{\eta \in E} 	m \left( r \dim(\eta)^\alpha + \dim(\eta)^2 O\left(r \log \frac{r}{m}\right) \right) + r\: m \cdot \dim(\eta)^\alpha.   
\end{dmath} 
\end{theorem}
 
Using Bratteli diagrams (see Section~\ref{subsection:Bratteli-diagrams}), we calculate the irreducible representations iteratively (see Section~\ref{section:FFT}), thus proving Theorem~\ref{theorem:TG-upper-bound}.

 \subsection{Summary of the sections} 
 In Section~\ref{section:background-notation}, we provide some background and notation. 
 We begin by defining wreath products of cyclic groups in Section~\ref{subsection:wreath-products}, 
 give an introduction to Clifford theory in Section~\ref{subsection:Clifford}, give a construction of $\vec{r}$-trees in Section~\ref{subsection:trees}, and then define Bratteli diagrams in Section~\ref{subsection:Bratteli-diagrams}. 
   
 In Section~\ref{section:irrep-iterated-wreath-products}, we prove Theorem~\ref{theorem:irreps-symm} and then give the number of irreducible representations for the iterated wreath product $\mathbb{Z}_{r_1}\wr \ldots \wr \mathbb{Z}_{r_k}$ in Theorem~\ref{theorem:irreducible-rep-recursion}. 
 In Section~\ref{section:branching-diagram-generalized-cyclic-rooted-tree-correspondence}, we prove the one-to-one correspondence between equivalence classes of irreducible representations of the generalized wreath product and orbits of compatible $\vec{r}|_{k}$-labels (Proposition~\ref{prop:one-to-one-equiv-irrep-orbits-labels}), give the number of $\vec r|_k$-trees of height $k$ in Corollary~\ref{cor:number-trees-height-k-recursion-Euler}, and write the dimension of an irreducible representation of an iterated wreath product of cyclic groups in terms of companion trees in Proposition~\ref{prop:dimension-rep-companion-tree}.
 In Section~\ref{section:FFT}, we prove Theorem~\ref{theorem:TG-upper-bound}, 
 generalizing Theorem 1 in \cite{MR1053228} by giving an explicit computation, and finally, in Section~\ref{section:conclusion-future}, we conclude by providing an open problem. 

  \subsection{Acknowledgment}    
The authors acknowledge Mathematics Research Communities for providing an exceptional working environment at Snowbird, Utah. They would like to thank Michael Orrison for helpful discussions, and the referees for extremely useful remarks on this manuscript. 
This manuscript was written during MSI's visit to the University of Chicago in 2014. She thanks their hospitality.  
   
\section{Background}\label{section:background-notation}
\subsection{Wreath products}\label{subsection:wreath-products}
We refer to Section 1.1 in \cite{MR2081042} for a beautiful exposition with illustrative examples about the construction of the wreath product $G\wr H$ of a finite group $G$ with a subgroup $H$ of $S_n$, which is summarized as follows. We define an action of $H$ on $G^n=G\times \cdots \times G$ by if $\pi\in H$ and $a=(a_1,a_2,\ldots, a_n)\in G^n$, then 
$\pi\cdot a := a^{\pi}=(a_{\pi^{-1}(1)},a_{\pi^{-1}(2)},\ldots, a_{\pi^{-1}(n)})$. 
The wreath product $G\wr H$ is defined to be $G^n\times H$ as a set, with multiplication given by 
\begin{equation}\label{eqn:multn-wreath-product}
(a;\pi)(b;\sigma) = (ab^{\pi};\pi\sigma). 
\end{equation}

Throughout this paper, we will fix $\vec{r} = (r_1,r_2,r_3,\ldots) \in \mathbb{N}^\omega$, a positive integral vector. We denote by $\vec{r}|_k :=(r_1,r_2,\ldots, r_k)$ the $k$-length vector found by truncating $\vec{r}$.   
\begin{definition}\label{definition:chain-of-subgroups}  
We define the {\em (generalized) 
$k$-th $\vec{r}$-cyclic wreath product} $W(\vec{r}|_k)$ recursively by:
	\begin{equation*}
		W(\vec{r}|_0) = \{ 0\} \text{ and } W(\vec{r}|_k) = W(\vec{r}|_{k-1}) \wr \mathbb{Z}_{r_k}.  
	\end{equation*}
\end{definition}

Note that multiplication for the wreath product $W(\vec{r}|_k)$ is defined recursively using \eqref{eqn:multn-wreath-product}. 

\begin{example}\label{example:iterated-wr-prod}
We have $W(\vec{r}|_1) = \mathbb{Z}_{r_1}$, $W(\vec{r}|_2) = \mathbb{Z}_{r_1} \wr \mathbb{Z}_{r_2}$, and $W(\vec{r}|_k) = \mathbb{Z}_{r_1} \wr \ldots \wr \mathbb{Z}_{r_k}$. 
\end{example}
Throughout this manuscript, we will be considering the chain of groups given in Definition~\ref{definition:chain-of-subgroups}.

\subsection{Clifford theory}\label{subsection:Clifford}
The following \cite{MR1038525}, \cite{MR1060103}, and \cite{MR2760311} contain an extensive background on Clifford theory, which allows one to recursively construct the irreducible representations of a group. 
In this manuscript, 
we will give a brief overview of the main results of Clifford theory. 

Let $G$ be a finite group and let $K$ be a normal subgroup of $G$. 
Then $G$ acts on the set of inequivalent irreducible representations of $K$. For any irreducible representation $\sigma$ of $K$, let $\Delta(\sigma)$ denote its orbit under this action, i.e., inequivalent conjugates of $\sigma$. 
Let $\Stab_G(\sigma)$ be the isotropy subgroup of $\sigma$ under the $G$-action. 
\begin{theorem}[Clifford theory]\label{thm:Clifford-theory}
Let $K$ be a normal subgroup of $G$. 
\begin{enumerate}
\item(\cite{MR0202859}, Theorem 10)\label{item:Clifford-one} 
If $\sigma$ be a representation of $K$, then 
\[ 
\Res_K^G \Ind_K^G \sigma = [\Stab(\sigma):K]\cdot \Delta(\sigma). 
\] 
\item(\cite{MR0202859}, Theorem 14)\label{item:Clifford-two} 
If $\rho:G\rightarrow GL(V)$ is an irreducible representation of $G$, then 
\[ 
\Res_K^G \rho =  \frac{d_{\rho}}{[G:\Stab(\sigma)]d_{\sigma}}\cdot \Delta(\sigma), 
\] 
where $\sigma$ is any irreducible representation of $K$ which appears in $\Res_K^G \rho$, and $d_{\rho}$ is the dimension of the vector space $V$.  
\end{enumerate}
\end{theorem}
We also call $d_{\rho}$ the {\em degree} of $\rho$. 

\subsection{Rooted trees of a fixed height}\label{subsection:trees}

We define $\vec{r}$-trees, a generalization of the $r$-trees in Section 3 of \cite{MR2081042}. 
A {\em rooted tree} is a connected simple graph with no cycles and with a distinguished vertex called the {\em root}. 
We say a node $v$ is in the {\em $j$-th layer} of a rooted tree if it is at distance $j$ from the root. 

\begin{definition}
We define the complete $\vec{r}$-tree, denoted by $T(\vec{r}|_k)$, of height $k$, or $\vec{r}|_k$-tree, recursively as follows: let $T(r_1)$ be the 1-layer tree consisting of a root node only. Let $T(\vec{r}|_2)$ consist of a root with $r_2$ children. Let $T(\vec{r}|_k)$ consist of a root node with $r_k$ children, 
with each the root of a copy of the $(k-1)$-layer tree $T(\vec{r}|_{k-1})$, which yields a tree with $k$ levels of nodes.  
\end{definition}

\begin{example}
The tree $T(r_1)$ is given by $\bullet$ and $T(r_1,r_2)$ with $r_2$ leaves is given by 
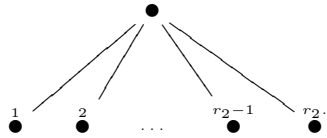
\tiny 
\[ 
\xymatrix@-1pc{ 
& & \ar@{-}[lldd] \ar@{-}[ldd]\bigcdot \ar@{-}[rdd] \ar@{-}[rrdd]& &  \\ 
& &  & &  \\ 
\stackrel{1}{\bigcdot} & \stackrel{2}{\bigcdot} & \ldots & \stackrel{r_2-1}{\bigcdot} &\stackrel{r_2.}{\bigcdot} \\ 
}
\] 
\end{example}

\begin{example}\label{example:complete-tree-3-layers-nodes}
Writing $\vec{r}|_3= (r_1,r_2,r_3)$, 
the following is the complete tree $T(\vec{r}|_3)$ of height $3$ with $3$ levels of nodes: 
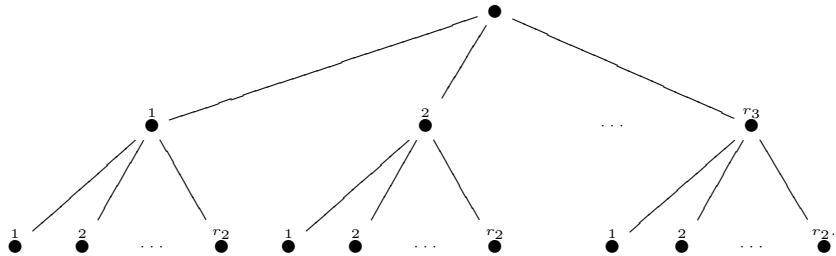
\tiny 
\[
\xymatrix@-1pc{
& & & & & & & \ar@{-}[llllldd]  \ar@{-}[ldd]\bigcdot  \ar@{-}[rrrrdd]& &  &  & & & &  \\ 
& & & & & & & &  & & & & & & \\ 
& & \ar@{-}[lldd] \ar@{-}[ldd] \stackrel{1}{\bigcdot}\ar@{-}[rdd] & & & &\ar@{-}[lldd] \ar@{-}[ldd]\stackrel{2}{\bigcdot} \ar@{-}[rdd]& & & \cdots&    &  \ar@{-}[lldd] \ar@{-}[ldd]  \stackrel{r_3}{\bigcdot} \ar@{-}[rdd] & & &  \\ 
& &  & & & & &  & & & & & & & \\
\stackrel{1}{\bigcdot} & \stackrel{2}{\bigcdot} & \cdots &\stackrel{r_2}{\bigcdot} &  \stackrel{1}{\bigcdot}&  \stackrel{2}{\bigcdot} & \cdots & \stackrel{r_2}{\bigcdot}&   & \stackrel{1}{\bigcdot} & \stackrel{2}{\bigcdot} & \cdots &\stackrel{r_2.}{\bigcdot}  & &  &\\
}
\] 
\end{example}

\begin{example}\label{ex:4-layers-node}
The complete tree $T(\vec{r}|_4)$ of height $4$ with $4$ levels of nodes is: 
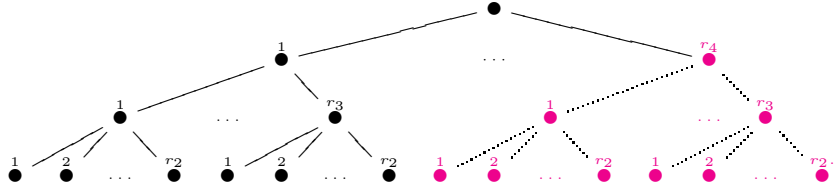
\tiny 
\[ 
\xymatrix@-1.5pc{
& & & & & & & & & \ar@{-}[dllll]   \bigcdot   \ar@{-}[drrrr]& & & &  & & & &    \\ 
& & & & & \ar@{-}[dlll] \stackrel{1}{\bigcdot} \ar@{-}[dr] & &  & &\cdots &  & & & \color{magenta}\ar@{..}@[magenta][dlll]\stackrel{r_4}{\bigcdot} \ar@{..}@[magenta][dr]\color{black}&  &  & &   \\ 
& & \ar@{-}[dll] \ar@{-}[dl]\stackrel{1}{\bigcdot}\ar@{-}[dr]&  &\ldots  & &  \ar@{-}[dll] \ar@{-}[dl]\stackrel{r_3}{\bigcdot}\ar@{-}[dr] &  & &  &  \ar@{..}@[magenta][dll] \ar@{..}@[magenta][dl]\color{magenta}\stackrel{1}{\bigcdot}\color{black} \ar@{..}@[magenta][dr]& & &\color{magenta}\ldots &  \ar@{..}@[magenta][dll] \ar@{..}@[magenta][dl] \color{magenta}\stackrel{r_3}{\bigcdot} \color{black}\ar@{..}@[magenta][dr]& &  &   \\ 
 \stackrel{1}{\bigcdot}& \stackrel{2}{\bigcdot}& \ldots & \stackrel{r_2}{\bigcdot} & \stackrel{1}{\bigcdot}& \stackrel{2}{\bigcdot}& \ldots &\stackrel{r_2}{\bigcdot}& \color{magenta}  \stackrel{1}{\bigcdot}\color{black}&\color{magenta} \stackrel{2}{\bigcdot}&\color{magenta} \ldots & \color{magenta}\stackrel{r_2}{\bigcdot} & \color{magenta}    \stackrel{1}{\bigcdot}&\color{magenta} \stackrel{2}{\bigcdot}& \color{magenta}\ldots &\color{magenta}\stackrel{r_2.}{\bigcdot} & &   \\
}
\] 
\end{example}

Notice that $T(\vec{r}|_k)$ has $\displaystyle{\prod_{i=2}^k } r_i$ leaves, with 
$\displaystyle{\prod_{i=k-j+1}^k r_i}$ nodes in the $j$-th layer. The subtree $T_v$ of $T=T(\vec{r}|_k)$ is the tree rooted at $v$ consisting of all the children and descendants of $v$. We call $T_v$ a maximal subtree of $T$ if $v$ is a child of the root, or equivalently if $v$ is in the second layer. 

\begin{example}
In Example~\ref{ex:4-layers-node}, the subtree indicated by dotted edges in magenta is a maximal subtree of  $T(\vec{r}|_4)$. 
\end{example}

\begin{definition}
Let  $V_{T(\vec{r}|_k)}$ be the set of vertices of the tree $T(\vec{r}|_k)$. 
	An {\em $\vec{r}|_k$-label} is a function $\phi: V_{T(\vec{r}|_k)} \rightarrow \mathbb{N}$ on the vertices of the tree $T(\vec{r}|_k)$ that assigns a natural number to each vertex. 
	A $\vec{r}|_k$-label is {\em compatible} if it satisfies the following: 
	\begin{enumerate}
	\item for $k=1$: $\phi( \text{root node}) \in [r_1]$,  and 
	\item for $k>1$:  
	\begin{enumerate}
	\item given any child of the root $v$, $\phi_v := \phi|_{T_v}$ is a compatible $\vec{r}|_{k-1}$-label, and 
	\item $\phi(\text{root node}) \in [d]$, where 	$\mathbb{Z}_d$ is the stabilizer of the action of $\mathbb{Z}_{r_k}$ on equivalence classes of $\{(\phi_v): v \text{ is a child of the root}\}$, 
	\end{enumerate}
	\end{enumerate}	
	where $\phi_v$ denotes the restriction of $\phi$ to the maximal subtree $T_v$. 
\end{definition}

We say that two compatible labels $\phi$ and $\psi$ of $T(\vec{r}|_k)$ are {\em equivalent}, and write $\phi \sim \psi$, if they are in the same orbit under the action of $W(\vec{r}|_k)$, or equivalently, if $\psi^{W(\vec{r}|_k)} = \phi^{W(\vec{r}|_k)}$, where $\psi^g(v) := \psi(v^g)$.

\subsection{Bratteli diagrams}\label{subsection:Bratteli-diagrams}

We refer to Section 4.1 in \cite{MR2081042} or to \cite{maslen2001cooley} for a detailed discussion on Bratteli diagrams. 

A {\em Bratteli diagram} $B$ is a weighted graph, which can be described by a set of vertices from a disjoint collection of sets $B_m$, $m\geq 0$, and edges that connect vertices in $B_m$ to vertices in $B_{m+1}$. Assuming that the set $B_0$ contains a unique vertex, the edges are labeled by positive integer weights. The set $B_m$ is the set of vertices at level $m$. If a vertex  $T_1\in B_m$ is connected to a vertex $T_2\in B_{m+1}$, then we write $T_1\leq T_2$. 

Given a tower of subgroups $\langle 1 \rangle=G_0\leq G_1 \leq \ldots \leq G_n$, the corresponding Bratteli diagram has vertices of set $B_i$ labeling the irreducible representations of $G_i$. 
If $\rho$ and $\eta$ are irreducible representations of $G_i$ and $G_{i-1}$, respectively, then the corresponding vertices are connected by an edge weighted by the multiplicity of $\eta$ in $\rho$ when restricted to $G_{i-1}$. 

\begin{example}\label{example:Bratteli-diagram}
The Hasse diagram of the partially order set with a labelling of the edges is the Bratteli diagram of the iterated wreath product of cyclic groups (see Section 4.1 in \cite{MR2081042} for an illustrated example for the $n$-fold iterated wreath product of $\mathbb{Z}_r$). 
\end{example}

\section{Irreducible representations of iterated wreath products}\label{section:irrep-iterated-wreath-products}

We will now prove Theorem~\ref{theorem:irreps-symm}.

\begin{proof}
First we note that $\mathcal{R}^{r_k} := \{ \rho_{x_1} \otimes \rho_{x_2} \otimes \cdots \otimes \rho_{x_{r_k}} : x \in [h]^{r_k} \}$ is a traversal for $W(\vec{r}|_{k-1})^{r_k}$. 
Consider the action of $\mathbb{Z}_{r_k}$ on $\mathcal{R}^{r_k}$ by its action on the indices, indexed by $[h]^{r_k}$. This is isomorphic to the action of $W(\vec{r}|_k) = W(\vec{r}|_{k-1})^{r_k} \rtimes \mathbb{Z}_{r_k}$ on $\widehat{W}(\vec{r}|_{k-1})^{r_k}$ by conjugation. 

Fix some $\sigma_x := \rho_{x_1} \otimes \cdots\otimes \rho_{x_{r_k}} \in \mathcal{R}^{r_k}$ (which corresponds to the word
 $x = x_1 \cdots x_{r_k} \in [h]^{r_k}$). The stabilizer of $x$ under the cyclic action of $\mathbb{Z}_{r_k}$ is a subgroup corresponding to 
 $$
 \mathbb{Z}_{d_x}  \cong \left(\frac{r_k}{d_x}\mathbb{Z}\right) \bigg/\left(r_k \mathbb{Z}\right)
 $$ 
 for some $d_x| r_k$. Notice that 
$W(\vec{r}|_{k-1})^{r_k} \trianglelefteq W(\vec{r}|_k)$.  In the language of Clifford theory, the inertia group for $\sigma_x$ is given by
	\begin{equation}
	I= I_{\sigma_x} = W(\vec{r}|_{k-1})^{r_k} \rtimes \mathbb{Z}_{d_x}. 
	\end{equation}
Also notice the inclusion $W(\vec{r}|_{k-1})^{r_k} \leq I \leq W(\vec{r}|_k)$ of a chain of subgroups. 
For $H\leq G$ and $\tau \in \widehat{H}$, denote 
\begin{equation}
\widehat{G}(\tau) = \left\{ \theta \in \widehat{G} : \tau \leq \Res_H^G \theta \right\}. 
\end{equation}  
In applying Clifford theory, we find that 
$\widehat{I}(\sigma) =  \{ \sigma \otimes \tau : \tau \in \widehat{\mathbb{Z}_d} \}$. More importantly, 
\begin{equation}
	\widehat{W}(\vec{r}|_k)(\sigma) =	\left\{ 
	\Ind_{W(\vec{r}|_{k-1})^{r_k}\rtimes \mathbb{Z}_{d_{\sigma}}}^{W(\vec{r}|_k)} \sigma \otimes \tau : \tau \in \widehat{\mathbb{Z}_d} \right\}
	\end{equation}
In addition, $\widehat{W}(\vec{r}|_k) 
= \displaystyle{ \bigcup_{\sigma \in \widehat{W}(\vec{r}|_{k-1})^{r_k}}} 
\widehat{W}(\vec{r}|_k)(\sigma)$. Also, if $\theta \in \widehat{W}(\vec{r}|_k)(\sigma)$ and $ \theta \in \widehat{W}(\vec{r}|_k)(\sigma')$ and $I^G_\sigma = I^G_{\sigma'}$, then there exists $g \in W(\vec{r}|_k)$ such that $\sigma = {\sigma'}^g$. 
The result follows. 
		\end{proof}

\begin{corollary}
For a particular $\sigma =  \rho_{x_1} \otimes \cdots \otimes \rho_{x_{r_k}}  \in \widehat{W}(\vec{r}|_{k-1})^{r_k}$ and $\tau \in \widehat{\mathbb{Z}_{d_x}}$, let $I$ be the inertia group of $\sigma$. 
Then we have 
	\begin{equation}
	\Res_ {W(\vec{r}|_{k-1})^{r_k}}^{W(\vec{r}|_k)} \left( \Ind_{ W(\vec{r}|_{k-1})^{r_k} \rtimes \mathbb{Z}_{d_x}}^ {W(\vec{r}|_k)} \sigma \otimes \tau \right) = \sigma \oplus \sigma^1 \oplus \ldots \oplus \sigma^{r/d_x} .
\end{equation}
	 
\end{corollary}

\begin{proof}
This follows from Theorem~\ref{theorem:irreps-symm} and an application of Clifford theory. 
\end{proof}

\subsection{Number of irreducible representations}\label{subsection:number-of-irrep}

Following the exact same argument for Theorem 2.2 in \cite{MR2081042}, we have the following recursion for the number of irreducible representations for $W(\vec{r}|_{k})$.
  
\begin{theorem}\label{theorem:irreducible-rep-recursion}
 
The number $M(\vec{r}|_{k} )$ of irreducible representations of $W(\vec{r}|_k)$ satisfies the recursion
  
\begin{equation}
M(\vec{r}|_k) = \frac{1}{r_k} \sum_{d|r_k} f(d) d^2 = \frac{1}{r_k} \sum_{d | c| r_k} \mu(c/d) M\left(\vec{r}|_{k-1} \right)^{r_k/c} d^2, 
\end{equation}
 where $M(\vec{r}|_1)=r_1$ and $\mu(n)$ is the Euler number for a natural number $n \in \mathbb{N}$. 
\end{theorem}


\section{Bijection between the branching diagram for generalized iterated wreath products and rooted trees}\label{section:branching-diagram-generalized-cyclic-rooted-tree-correspondence}
 
In this section, we will give a combinatorial structure describing the branching diagrams for the iterated wreath products of cyclic groups. 
In a similar spirit to Proposition 4.6 in \cite{MR2081042}, we have the following: 
\begin{proposition}\label{prop:one-to-one-equiv-irrep-orbits-labels}
There exists a one-to-one correspondence between equivalence classes of irreducible representations of 
the wreath product $W(\vec{r}|_k)$ of cyclic groups 
and orbits of compatible $\vec{r}|_k$-labels. 
\end{proposition} 

\begin{proof}
We inductively define a map $F: \widehat{W}(\vec{r}|_k) \rightarrow \{ \vec{r}|_k \text{-labels}\}$ that gives a bijection between equivalence classes of irreducible representations of $W(\vec{r}|_k)$ and orbits of labels. 
Denote by $z_1$ a fixed generator of $\mathbb{Z}_{r_1}$ and $w_1$ a fixed $r_1$-th root of unity. For $i=1 , \ldots , r_1$, let $\tau^{(1)}_i$ denote the irreducible representation of $\mathbb{Z}_{r_1}$ such that $\tau(z_1) = (w_1)^j$. 

Recall that  $T(\vec{r}|_1)$ consists of only a root node, and all $\vec{r}|_1$-labels $\phi: V_{T(\vec{r}|_1)} \rightarrow [r_1]$ are determined exactly by their value on the root node. Thus, $\widehat{W}(\vec{r}|_1) = \{ \tau^{(1)}_j : j \in [r_1] \}$ is clearly in bijection with $\vec{r}|_1$-labels. To be precise, $F: \widehat{W}(\vec{r}|_1) = \widehat{\mathbb{Z}_{r_1}} \rightarrow \{ \vec{r}|_1 \text{-labels} \}$ is defined by $F(\tau^{(1)}_j)(\text{root}) = j$ for all $j \in [r_1]$. 

Suppose for induction that $F: \widehat{W}(\vec{r}|_{k-1}) \rightarrow \vec{r}|_{k-1} \text{-labels}$ gives a bijection of the form desired. We define $F: \widehat{W}(\vec{r}|_{k}) \rightarrow \{ \vec{r}|_{k} \text{-labels} \}$ as follows. 

 Let $\{ \rho_1 , \ldots , \rho_h \}$ be an enumeration for a traversal $\mathcal{R}_{W(\vec{r}|_{k-1})}$ for $W(\vec{r}|_{k-1})$. It suffices to define $F$ on a traversal for $\widehat{W}(\vec{r}|_k)$, such as $\mathcal{R}_{W(\vec{r}|_k)}$ defined above. Then, let $\rho_{x_1} \otimes \cdots \otimes \rho_{x_{r_k}} \otimes \tau $ be an arbitrary element of $\mathcal{R}_{W(\vec{r}|_k)}$. Let 
 $$
 \phi := F\left(\Ind_{W(\vec{r}|_{k-1}) \wr \mathbb{Z}_{d_x}}^{W(\vec{r}|_k)} \rho_{x_1} \otimes \cdots \otimes \rho_{x_{r_k}} \otimes \tau  \right) : V_{T(\vec{r}|_k)} \rightarrow \mathbb{N}
 $$ 
 be defined as follows: 
let $U = \{ u_1, \ldots , u_{r_k} \}$ be the set of $r_k$ children of the root node. 
For any $u\in U$, let $\phi_u$ denote the vector $\vec{r}|_{k-1}$ found by restricting $\phi$ to the maximal subtree $T_u$. Let $d = d_x$ be as defined in \eqref{eqn:d_x}. 
Let $\zeta$ be a fixed generator of $\mathbb{Z}_d$ and $\omega$ a fixed $d$-th root of unity. For $i = 1 , \ldots, d$, let $\pi_i$ denote the irreducible representation of $\mathbb{Z}_d$ such that $\pi(\zeta) = \omega^i$. 

 \begin{itemize}
 \item For any non-root node of $T(\vec{r}|_k)$ we let the value of $\phi$ be defined to satisfy $\phi_{u_i} := F(\rho_{x_i})$. 
 \item For the root node, notice that $\tau \in \widehat{\mathbb{Z}_{d_x}}$ by definition of $\mathcal{R}_{W(\vec{r}|k)}$. But $d_x$ is exactly the integer such that $\mathbb{Z}_{d_x}$ is the stabilizer of $\mathbb{Z}_{r_k}$ on the equivalence classes of $\{ \phi_u : u\in U\}$, so $\tau \in  \widehat{\mathbb{Z}_{d}}$. Thus $\tau = \pi_j$ for some $j \in [d_x]$. Let $\phi (\text{root}) := j$. 
 \end{itemize}
 
It follows from induction that $F$ is a bijection from the equivalence classes of $\widehat{W}(\vec{r}|_k)$ to the orbits of $\vec{r}|_k$ labels. 

\end{proof}

\begin{corollary}\label{cor:number-trees-height-k-recursion-Euler}
The number $h_k(\vec{r}|_k)$ of $\vec{r}|_k$-trees of height $k$ is given by the recursion 
\[  
h_k(\vec{r}|_k) = \dfrac{1}{r_k}\sum_{d|c|r_k} \mu(c/d)h_{k-1}(\vec{r}|_{k-1})^{r_k/c} d^2,  
\]  
where $\mu(n)$ is the Euler number of $n \in \mathbb{N}$. 
\end{corollary}

\subsection{Degrees of irreducible representations}\label{section:degrees}

Following the discussion in Section 4.1.1 in \cite{MR2081042}, we define for any $\vec{r}|_k$-label $\phi$ the companion tree $C_\phi$. 

\begin{definition}
Fix an $\vec{r}|_k$-label $\phi$. Define the {\em companion label} $C_\phi$ to be the $\vec{r}|_k$-label $C: V_{T(\vec{r}|_k)} \rightarrow \mathbb{N}$ as follows: an arbitrary vertex $v$ on the $\ell$-th layer of the complete $\vec{r}$-tree $T(\vec{r}|_k)$ is 

\begin{equation*}
C(v) = \left |\left\{  
\phi_u^{ W(\vec{r}|_{k-\ell-1}) } : u \text{ is a child of } v 
\right\} \right|. 
\end{equation*} 

\end{definition}

Recall that $x^G=\{x^g:g\in G \}$, the orbit of $x$ under the action of $G$. So $C(v)$ is the number of orbits occupied by the $\vec{r}|_{k-\ell-1}$-labels of the $r_{k-\ell-1}$ maximal $\vec{r}|_{k-\ell-1}$-subtrees of $T_v$, or the number of inequivalent sublabels on maximal subtrees of $T_v$ given by $\phi$.

\begin{example}
The companion tree $C_{\phi}$ to Example~\ref{example:complete-tree-3-layers-nodes} is 
\tiny 
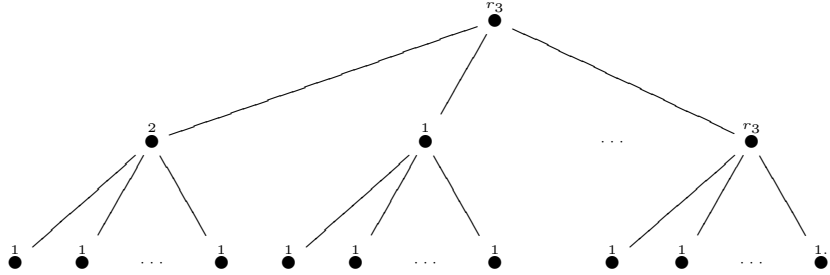
 \[
\xymatrix@-1pc{
& & & & & & & \ar@{-}[llllldd]  \ar@{-}[ldd]\stackrel{r_3}{\bigcdot} \ar@{-}[rrrrdd]& &  &  & & & &  \\ 
& & & & & & & &  & & & & & & \\ 
& & \ar@{-}[lldd] \ar@{-}[ldd] \stackrel{2}{\bigcdot}\ar@{-}[rdd] & & & &\ar@{-}[lldd] \ar@{-}[ldd]\stackrel{1}{\bigcdot} \ar@{-}[rdd]& & & \cdots&    &  \ar@{-}[lldd] \ar@{-}[ldd]  \stackrel{r_3}{\bigcdot} \ar@{-}[rdd] & & &  \\ 
& &  & & & & &  & & & & & & & \\
\stackrel{1}{\bigcdot} & \stackrel{1}{\bigcdot} & \cdots &\stackrel{1}{\bigcdot} &  \stackrel{1}{\bigcdot}&  \stackrel{1}{\bigcdot} & \cdots & \stackrel{1}{\bigcdot}&   & \stackrel{1}{\bigcdot} & \stackrel{1}{\bigcdot} & \cdots &\stackrel{1.}{\bigcdot}  & &  &\\
}
\]
\end{example}

Similar to Proposition 4.3 in \cite{MR2081042}, we obtain: 
\begin{proposition}\label{prop:dimension-rep-companion-tree}

Let $\rho$ be an irreducible representation of $W(\vec{r}|_k)$ associated to $\vec{r}|_k$-tree $T$ with companion tree $C_T$. Then the dimension $d_\rho$ of $\rho$ is given by 

        \begin{equation}
        d_\rho = \prod_{v} C(v), 
        \end{equation}
the product of the value of the companion label $C$ on all vertices. 
\end{proposition}


\section{Fast Fourier transforms, adapted bases and upper bound estimates}\label{section:FFT}   

We will now prove Theorem~\ref{theorem:TG-upper-bound}.

\begin{proof}
Enumerate $\mathcal{R}_K = \{ \eta_1, \ldots, \eta_L \}$ so that $K$ has $L$ inequivalent irreducible representations. Notice that 
$$
\mathcal{R}_{K^r} = \{ \eta_{\ell_1} \otimes \cdots \otimes \eta_{\ell_r} : \ell_1, \ldots , \ell_r \in [L] \} = \{ \eta_{\vec{\ell}} : \vec{\ell} \in [L]^r\}.
$$ 

Notice that $\{(1, i): i \in [r]\}$ is a complete set of coset representatives of $K^r$ in $K \wr \mathbb{Z}_r = K^r \rtimes \mathbb{Z}_r$. Let $f_i:K^r \rightarrow \mathbb{C}$ be defined by $f_i(\vec{k}) = f(\vec{k}, i)$. 
Calculating $ \{ \widehat{f_i}(\eta_{\vec{\ell}}) : i \in [r], \vec{\ell} \in [L]^r\}$ requires $r \cdot T(K^r)$ computations. We will use the $\widehat{f_i}$'s to compute the Fourier transform $\widehat{f}$. 

It suffices to compute $\widehat{f}$ for every induced irreducible representation of maximal extensions for all irreducible representations of $K^r$. 
So fix some $\vec{\eta} = \eta_{\vec{\ell}} \in \mathcal{R}_{K^r}$   
corresponding to some fixed $\vec{\ell} \in [L]^r$. Let $m = m_{\vec{\ell}} := \min \{ i\neq 0 : \vec{\ell}^i = \vec{\ell} \}$, where $(\ell_1 \ell_2 \cdots \ell_r)^i := \ell_{i+1} \cdots \ell_r \ell_1 \cdots \ell_i$. Identify $K \rtimes \mathbb{Z}/r\mathbb{Z}$ with $K \rtimes (\mathbb{Z}/ m\mathbb{Z} \times m\mathbb{Z}/r\mathbb{Z})$ to relabel the element $(\vec{k}, i)$ by $(\vec{k}, i_1, i_2) = (\vec{k}, i\mod m, \lfloor i/m \rfloor)$.  

The inertia group of $\vec{\eta}$ is given by $H = K^r \rtimes \mathbb{Z}_{r/m} =  K^r \rtimes ( 1 \times m\mathbb{Z}/ r\mathbb{Z} )$. Fix some $(r/m)$-th root of unity $\zeta$. Denote the irreducible representations of $H/K^r \cong \mathbb{Z}_{r/m}$ by $\chi_j$ for $j\in [r/m]$, where $\chi_j(i) = \zeta^{ij}$. Let $\widetilde{\chi}: H \rightarrow \mathbb{C}$ be given by $\widetilde{\chi}(\vec{k}, i_2) = \chi(i_2)$. If $\widetilde{\eta}$ is one extension of $\vec{\eta}$ to $H$, then the set $\{ \widetilde{\chi_j} \otimes \widetilde{\eta} : j \in [r/m] \}$ gives the complete set of inequivalent extensions of $\vec{\eta}$ to $H$. All these extensions are irreducible. 

Similarly, we relabel $\{ f_i : K^r \rightarrow \mathbb{C} : i \in [r] \}$ by $\{ f_{i_1,i_2} : K^r \rightarrow \mathbb{C} : i_1 \in [m], i_2 \in [r/m] \}$ where $ f_{i_1, i_2} (\vec{k}) = f( \vec{k}, i_1, i_2) $. For all $i_1 \in [m]$, let $f_{i_1} : H \rightarrow \mathbb{C}$ be given by $f_{i_1}(\vec{k}, i_2) = f(\vec{k}, i_1, i_2)$. 
For all $i_1$, we have 

\begin{eqnarray*}
\widehat{f}_{i_1} ( \widetilde{\chi} \otimes \widetilde{\eta}) & = & \sum\limits_{\vec{k} \in K^r, i_2 \in [r/m]} \widetilde{\chi}(i_2) \cdot \widetilde{\eta}(\vec{k}, i_2) \cdot f_{i_1}(\vec{k},  i_2) \\ 
& = & \sum\limits_{\vec{k} \in K^r, i_2 \in [r/m]}  \widetilde{\chi}(i_2) \cdot  \widetilde{\eta}\left( (\vec{1}, i_2) (\vec{k}, 1) \right) f_{i_1, i_2} (\vec{k}) \\ 
& = & \sum\limits_{i_2 \in [r/m]}  \widetilde{\chi}(i_2) \widetilde{\eta}(\vec{1}, i_2) \sum\limits_{\vec{k} \in K^r} \eta(\vec{k}) f_{i_1, i_2} (\vec{k}) \\ 
& = & \sum\limits_{i_2 \in [r/m]} \widetilde{\chi}(i_2) \widetilde{\eta}(\vec{1}, i_2) \widehat{f}_{i_1, i_2} (\eta). 
\end{eqnarray*}

Notice that since $\widetilde{\eta}$ is an extension of $\eta$, $\widetilde{\eta}(\vec{k}, 1) = \eta(\vec{k})$ for any $\vec{k} \in K^r$. Let $\alpha_{i_1}(i_2) = \widetilde{\eta}(\vec{1}, i_2) \widehat{f}_{i_1, i_2} (\eta)$. Since $\widehat{f}_{i_1, i_2} (\eta)$ and $\widetilde{\eta}$ are $\dim(\eta) \times \dim(\eta)$ matrices, finding $\alpha_{i_1}(i_2)$ for all $i_1,i_2$ requires $m \cdot r/m \cdot \dim(\eta)^\alpha$ computations, where $\alpha$ is the constant of matrix multiplication. Further, the $(j,k)$-th entry of the resulting matrix of the summation is the Fourier transform of $\alpha_{jk}$ at the irreducible $\chi$. As the associated group is $\mathbb{Z}_{r/m}$, Fourier transforms require time $O(\frac{r}{m} \log \frac{r}{m})$. There are $\dim(\eta)^2$ of these functions, so computation of the final matrix for all $i_1$ requires 
$$
m \cdot \dim(\eta)^2 \cdot O\left(\frac{r}{m} \log \frac{r}{m}\right) 
= \dim(\eta)^2 \cdot O\left(r \log \frac{r}{m}\right)
$$ 
computations.

Now, we are interested in computing the matrix value of the transform on the induced representation $\rho : = \Ind_H^G \widetilde{\chi} \otimes \widetilde{\eta}$ for a specific $\widetilde{\chi} \otimes \widetilde{\eta}$. However, $\rho(\vec{k}, i_1, 1)$ is a block diagonal matrix with entries $\left( \widetilde{\chi} \otimes \widetilde{\eta} \right)^{(j)}$ in the $j$-th place, where the $j$ represents the irreducible found by conjugating by $(\vec{1}, 1, j)$. Using this, we obtain:   

\begin{eqnarray*}
	 \widehat{f}(\rho) & =  & \sum\limits_{\vec{k}, i_1, i_2} \rho(\vec{k}, i_1, i_2) \cdot f_{i_1, i_2}(\vec{k}) \\
	 & = & \sum\limits_{i_2} \rho(\vec{1}, 1, i_2) 
	 \sum\limits_{\vec{k}, i_1}  
	 \left( \begin{array}{ccc}
	 (\widetilde{\chi} \otimes \widetilde{\eta})^{(1)} (\vec{k}, i_1) & \cdots &  0 \\ 
	 \vdots & \ddots & \vdots  \\ 
	 0 & \cdots & (\widetilde{\chi} \otimes \widetilde{\eta})^{(m)}(\vec{k}, i_1)
	 \end{array}  \right) f_{i_1, i_2} (\vec{k})  \\ 
& = & \sum\limits_{i_2} \rho(\vec{1}, 1, i_2) 
	\left( \begin{array}{ccc}
	\widehat{f}_{i_2}(\widetilde{\chi} \otimes \widetilde{\eta}^{(1)} ) & \cdots & 0 \\ 
	\vdots & \ddots  & \vdots \\ 
	0 & \cdots & \widehat{f}_{i_2}(\widetilde{\chi} \otimes \widetilde{\eta}^{(m)} )
	\end{array} 
	\right).
\end{eqnarray*}	
With $\dim(\eta)^\alpha \cdot m^2$ computations for each induced representation and a total of $\frac{r}{m}$ induced representations, 
we have 
$r \cdot m \dim(\eta)^\alpha$ computations for each orbit of irreducible representations of $K^r$ under conjugation by $G$. 

In total, we need
	\begin{equation}
	r \cdot T(K^r) +	\sum\limits_{\eta \in E} 	m \left( r \dim(\eta)^\alpha + \dim(\eta)^2 O\left(r \log \frac{r}{m}\right) \right) + r\: m \cdot \dim(\eta)^\alpha  
	\end{equation}
computations, where $E$ is a set of representatives for each $G$-orbit of irreducible representations of $K^r$. 
The size of $E$ is given simply by the number of orbits of $[L]^r$ under action by $\mathbb{Z}_r$. 
\end{proof}


\section{Conclusion and future direction}\label{section:conclusion-future}
We have given an explicit description of a traversal for the iterated wreath product $W(\vec{r}|_k)$ of cyclic groups in Theorem~\ref{theorem:irreps-symm}, and we have determined a tighter upper bound of the FFT computation time for the iterated wreath product in Theorem~\ref{theorem:TG-upper-bound}. 

In our sequel manuscript \cite{Branching-diag-symm-Im-Wu}, we examine the representation theory of generalized iterated wreath products of symmetric groups, where we give a complete description of the traversal for these families of generalized iterated wreath products, and show the existence of a bijection between equivalence classes of irreducible representations of the generalized iterated wreath product and orbits of labels on certain rooted trees.  

We conclude with an open problem, which is to find adapted bases and fast Fourier transform operation bounds for chains of subgroups of iterated wreath products of more general classes of groups.

\appendix

\newcommand{\etalchar}[1]{$^{#1}$}
\def\cprime{$'$} \def\cprime{$'$} \def\cprime{$'$} \def\cprime{$'$}
\providecommand{\bysame}{\leavevmode\hbox to3em{\hrulefill}\thinspace}
\providecommand{\MR}{\relax\ifhmode\unskip\space\fi MR }
\providecommand{\MRhref}[2]{%
  \href{http://www.ams.org/mathscinet-getitem?mr=#1}{#2}
}
\providecommand{\href}[2]{#2}

\end{document}